\documentclass{amsart}

\title[Multilinear Kakeya via Borsuk--Ulam]{The Endpoint Multilinear
  Kakeya theorem via the Borsuk--Ulam theorem}
\author{Anthony Carbery and Stef\'an Ingi Valdimarsson}
\thanks{}
\address{Anthony Carbery, 
School of Mathematics and Maxwell Institute for Mathematical Sciences, 
University of Edinburgh,
JCMB, 
King's Buildings, 
Mayfield Road, 
Edinburgh, EH9 3JZ, 
Scotland.} 
\email{A.Carbery@ed.ac.uk}

\address{Stef\'an Ingi Valdimarsson,
Science Institute,
University of Iceland,
Dunhagi 3,
107 Reykjavik, 
Iceland.} 
\email{siv@hi.is}

\date{22nd May 2012}
\usepackage{amssymb}
\usepackage{amsmath}
\usepackage{url}

\newtheorem{theorem}{Theorem}
\newtheorem{proposition}{Proposition}
\newtheorem{lemma}{Lemma}

\begin{document}

\begin{abstract}
We give an essentially self-contained proof of Guth's recent endpoint multilinear
Kakeya theorem which avoids the use of somewhat sophisticated
algebraic topology, and which instead appeals to the Borsuk--Ulam theorem. 
\end{abstract}

\maketitle
\section{Introduction}
\noindent
The multilinear Kakeya problem was introduced in \cite{B}, and its
study began in earnest in \cite{BCT}, where the natural conjecture was established 
up to the endpoint.
Working in $\mathbb{R}^n$, we suppose that we are given $n$ transverse families 
$\mathcal{T}_1, \dots, \mathcal{T}_n$ of $1$-tubes, which means that each $T \in \mathcal{T}_j$ 
is a $1$-neighbourhood of a doubly-infinite line in $\mathbb{R}^n$ with direction 
$e(T) \in \mathbb{S}^{n-1}$, and that the directions $e(T)$ for $T \in \mathcal{T}_j$ all lie within a 
small fixed neighbourhood (depending only on the dimension $n$) of the $j$'th standard basis vector $e_j$. 

\medskip
\noindent
The question is whether for each $q \geq 1/(n-1)$ we have the inequality
\begin{equation*}\label{MKC}
\int_{\mathbb{R}^n} \left(\sum_{T_{1} \in \mathcal{T}_1} a_{T_1} \chi_{T_1}(x) \dots 
\sum_{T_{n} \in \mathcal{T}_n} a_{T_n} \chi_{T_n}(x) \right)^q \; dx
\leq C_{n,q} \left(\sum_{T_{1} \in \mathcal{T}_1} a_{T_1} \dots 
\sum_{T_{n} \in \mathcal{T}_n} a_{T_n}\right)^q
\end{equation*}
for nonnegative coefficients $a_{T_j}$. In \cite{BCT} this was proved for each $q > 1/(n-1)$ using a 
heat-flow technique which, because of certain error terms arising, did not apply at the endpoint 
$q = 1/(n-1)$. (For further background on this problem consult \cite{BCT}.)

\medskip
\noindent
More recently, Guth in \cite{G} established the endpoint case $q = 1/(n-1)$ using completely different 
techniques motivated in part by the polynomial method used by Dvir \cite{Dv} to solve the finite field 
Kakeya set problem, but which also relied upon a fairly heavy dose of algebraic topology, and which were 
therefore perhaps a little intimidating to the analyst or combinatorialist. In particular, Guth used the 
technology of cohomology classes, cup products and the Lusternik--Schnirelmann vanishing lemma in 
establishing his result. We believe
that the endpoint multilinear Kakeya theorem is of such significance and importance that a proof free of 
these techniques should be available, and so the purpose of this paper is to provide an argument leading to
Guth's result which does not rely upon such sophisticated algebraic topology, but whose input is instead the 
Borsuk--Ulam theorem. It is hoped therefore that this paper might lead
to further exploitation of Guth's techniques in a more flexible
setting. (For some recent works applying the multilinear
Kakeya perspective in other contexts, see \cite{B1}, \cite{B2} and \cite{B3}.) 

\medskip
\noindent
The Borsuk--Ulam theorem, while topological in nature, nevertheless has many proofs accessible 
to the analyst -- see for example \cite{Mat}, and also \cite{C} for a recent such proof. (See Section 
\ref{sec:BU} below for its statement.) The use of the Borsuk--Ulam theorem in the context of Kakeya theorems is 
by now natural, as it can be considered as a topological analogue of the elementary linear-algebraic statement that 
there are no linear injections $T: V \to W$ if $V$ and $W$ are finite-dimensional vector spaces 
with dim $V > $ dim $W$; this was a key element of Dvir's solution \cite{Dv} of the finite field 
Kakeya problem. It also features explicitly in Guth's warm-up discussion to the full result of \cite{G}.

\medskip
\noindent
In order to proceed, 
we place matters in a more general context which does not impose conditions on the directions 
of the tubes, nor requires the level of multilinearity to equal the dimension of the underlying 
euclidean space. Thus we now suppose that 
we are given $d$ arbitrary families of $1$-tubes $\mathcal{T}_1, \dots, \mathcal{T}_d$ in $\mathbb{R}^n$, 
where $d \leq n$. For $v_1, \dots v_d \in \mathbb{R}^n$ let $v_1 \wedge \cdots \wedge 
v_d$ denote the unsigned (i.e. nonnegative) $d$-dimensional volume of the parallelepiped whose sides are 
given by the vectors $v_1, \dots, v_d$.

\begin{theorem}[The Multilinear Kakeya Theorem]\label{main}
Let $2 \leq d \leq n$. Then there exists a constant $C_{d,n}$ such that if $\mathcal{T}_1, \dots, 
\mathcal{T}_d$ are families of $1$-tubes in $\mathbb{R}^n$, we have
\begin{equation}
\begin{aligned}
\int_{\mathbb{R}^n} & \left(\sum_{T_{1} \in \mathcal{T}_1} a_{T_1} \chi_{T_1}(x) \dots 
\sum_{T_{d} \in \mathcal{T}_d} a_{T_d} \chi_{T_d}(x) \; e(T_1) \wedge \dots \wedge e(T_d) \right)^{1/(d-1)} 
\; dx \\
& \leq C_{d,n} \left(\sum_{T_{1} \in \mathcal{T}_1} a_{T_1} \dots 
\sum_{T_{d} \in \mathcal{T}_d} a_{T_d}\right)^{1/(d-1)}.
\end{aligned}
\end{equation}
\end{theorem}

\medskip
\noindent
(The case $d=2$ is of course trivial.) 

\medskip
\noindent
The situation where the level of multilinearity is less than the ambient euclidean dimension was already 
addressed in \cite{BCT}, where once again the result was established up to the endpoint. The incorporation 
of the factor $e(T_1) \wedge \dots \wedge e(T_d)$ on the left-hand side is natural in view of the
affine-invariant formulation of the Loomis--Whitney inequality, and was considered 
in Section 7 of \cite{BG}, where Theorem \ref{main} was first proved. Indeed, when $d=n$, the statement of 
Theorem \ref{main} is affine-invariant.\footnote{The multilinear Kakeya theorem can also be cast in the following equivalent form when $d=n$. For a unit vector 
$\omega \in \mathbb{R}^n$ let $\Pi_\omega$ denote the hyperplane in $\mathbb{R}^n$ which is 
perpendicular to $\omega$ and which contains the origin. Let $\pi_\omega : \mathbb{R}^n \to \Pi_\omega$
be the orthogonal projection map. Then for nonnegative $g_j$ we have
$$ \int_{\mathbb{R}^n} \left(\int_{\mathbb{S}^{n-1}} \dots \int_{\mathbb{S}^{n-1}}
\left(\prod_{j=1}^n g_j(\omega_j, \pi_{\omega_j} x)\right)\omega_1 \wedge \dots \wedge \omega_n \; d\sigma(\omega_1) \dots d \sigma(\omega_n)\right)^{1/(n-1)} dx $$
$$\leq C_{n} \prod_{j=1}^n \left(\int_{\mathbb{S}^{n-1}} \int_{\Pi{_{\omega_j}}} g_j(\omega_j, \xi) 
d\xi d \sigma(\omega_j)\right)^{1/(n-1)}.$$} 
A variant of Theorem \ref{main} 
where lines are replaced by algebraic curves of bounded degree was also proved in \cite{BG} (and can 
likewise be established by replacing Guth's original argument for Theorem \ref{algtop0} below by 
that of the current paper). On the other hand, the results of \cite{BCT} have a somewhat more general scope
in so far as they apply to $1$-neighbourhoods of $k$-planes for arbitrary $k$, rather than just
$1$-neighbourhoods of lines, i.e. tubes, as in the present discussion.

\medskip
\noindent
The principal notion that Guth employs in proving the endpoint theorem is that of the 
{\em visibility}   vis$\,(Z)$ of a hypersurface $Z \subseteq \mathbb{R}^n$ -- see Section 
\ref{sect:vis} below for the definition, which differs from Guth's in
so far as in our treatment it (roughly)
scales as does $(n-1)$-dimensional Hausdorff measure $\mathcal{H}_{n-1}$ -- 
and the centrepiece of Guth's argument is the following result:

\begin{theorem}\label{algtop0}
Given a nonnegative function
$M$ defined on the lattice $\mathcal{Q}$ of unit cubes of $\mathbb{R}^n$,
there exists a non-zero polynomial $p$ such that
$$ {\rm{deg }} \; p \leq C_n \left(\sum_{Q \in \mathcal{Q}} M(Q)^n\right)^{1/n}$$
and such that if we set $Z = Z_p = \{x \in \mathbb{R}^n \; : \; p(x) =
0\}$, then for all $Q \in \mathcal{Q}$ we have 
$$ {\rm{vis }}\; (Z \cap Q) \geq C_n M(Q).$$
\end{theorem}

\medskip
\noindent
It is in the proof of this result that Guth uses algebraic-topological techniques, and 
the main contribution of the present paper is to provide a proof of Theorem \ref{algtop0} which 
does not use such topological machinery, but is instead a consequence
of the Borsuk--Ulam theorem. (In fact, in our proof of Theorem
\ref{algtop0}, we do not use the Borsuk--Ulam theorem {\em per se} but
instead an equivalent Lusternik--Schnirelmann type covering statement. See Section \ref{sec:BU}.)
On the other hand, we must acknowledge that many of the arguments and
constructions of the present paper are inspired by those of Guth's approach.

\medskip
\noindent
In view of the connection between visibility and $(n-1)$-dimensional Hausdorff measure, and as a warm-up 
to our proof of Theorem \ref{algtop0}, we indicate how the Borsuk--Ulam theorem can be used to 
establish the following morally weaker variant of Theorem \ref{algtop0}.

\begin{proposition}\label{warmup}
Given a finitely supported function $M$ defined on the lattice $\mathcal{Q}$
of unit cubes of $\mathbb{R}^n$ and taking nonzero 
values
in $[1, \infty)$, there exists a
non-zero polynomial $p$ such that 
$$ {\rm{deg }} \; p \leq C_n \left(\sum_{Q\in\mathcal{Q}} M(Q)^n\right)^{1/n}$$
and such that for all $Q\in\mathcal{Q}$ 
$$ \mathcal{H}_{n-1} \; (Z \cap Q) \geq C_n M(Q).$$
\end{proposition}

\begin{proof} Break up each $Q$ into $\sim M(Q)^n$ congruent subcubes $S$; note that 
altogether we have $\sim \sum_Q M(Q)^n$ small cubes $S$ of various sizes. Consider the map 
$$ F: p \mapsto \left\{\int_{\{p > 0\} \cap S} 1  -\int_{\{p < 0\} \cap S} 1 \right\}_{S}$$
defined on the vector space $\mathcal{P}_k$ of polynomials of degree at most $k$ in $n$ real
variables, which has dimension $\sim k^n$. Clearly $F$ is continuous, homogeneous of degree 
$0$ and odd.

\medskip
\noindent
So we can think of $F$ as 
$$F : \mathbb{S}^N\rightarrow \mathbb{R}^J $$
where $N \sim k^n$ and $J \sim \sum_Q M(Q)^n$.

\medskip
\noindent
So provided $N \geq J$ -- which we can arrange if $k \sim \left(\sum_Q
M(Q)^n\right)^{1/n}$ --  the Borsuk--Ulam theorem tells us that 
$F$ vanishes at some $p$. This means that the zero set $Z$ of $p$ exactly bisects each $S$. 

\medskip
\noindent
Now if $S$ is a subcube of $Q$, $S$ will have volume $\sim M(Q)^{-n}$ and
diameter $\sim M(Q)^{-1}$ and hence any bisecting surface will meet
it in a set of $(n-1)$-dimensional measure $\gtrsim
M(Q)^{-(n-1)}$. This will be true for each of the $M(Q)^n$ disjoint $S$'s whose union
is $Q$, so $Z$ will meet $Q$ in a set of $(n-1)$-dimensional measure $\gtrsim M(Q)^n \times
M(Q)^{-(n-1)} = M(Q)$, as was needed.
\end{proof}

\medskip
\noindent
In the proof we used the ``geometrically obvious'' fact that a
hypersurface bisecting the unit cube must have large surface area
inside the cube. For a discussion of this in the context of the unit
ball, see Lemma \ref{bisect} in the Appendix. Note that in the statements of both Theorem \ref{algtop0} 
and Proposition \ref{warmup}, a polynomial has the desired properties if and only if
any non-zero scalar multiple of it does; for this reason we may choose to search for a suitable 
polynomial within the unit sphere of the class of polynomials of a given degree. 

\medskip
\noindent
Proposition \ref{warmup} is morally weaker than Theorem \ref{algtop0} because not only does it place stronger 
conditions on $M$, but more importantly, in many situations of interest, we have 
${\rm{vis }}\; (Z \cap Q) \leq C_n \mathcal{H}_{n-1} \; (Z \cap Q)$ -- see \eqref{geomxx} below.

\medskip
\noindent
On an informal level, the fundamental difference between the proof of Theorem \ref{algtop0} and that of Proposition 
\ref{warmup} is that, roughly speaking, we no longer chop each cube $Q$ into $\sim M(Q)^n$ congruent 
{\em subcubes}, but we instead select, for each $Q$, an {\em ellipsoid} 
$E(Q)$ of volume $\sim M(Q)^{-n}$, so that $\sim M(Q)^n$ translates of $E(Q)$ essentially tessellate $Q$. 
However the {\em shape} and {\em orientation} of the ellipsoid $E(Q_0)$ will depend not only on the 
value of $M(Q_0)$ but on the whole ensemble $\{M(Q)\}_Q$, and is in effect an output of the 
Borsuk--Ulam theorem at the same time as it produces the desired 
polynomial. At the risk of over-simplifying matters, 
we now give an informal example which illustrates why, if we want the broad thrust of the proof of Proposition \ref{warmup}
to work in the context of Theorem \ref{algtop0}, the shape of the ellipsoid selected {\em must} depend on the totality of 
the function $M(Q)$. This example may be safely ignored on a first reading of the paper.

\medskip
\noindent
{\bf Informal example.} Let $n=2$ and consider the function $M(Q)$ which is supported on a row of 
$N$ unit cubes centred at $(k-1/2, 1/2)$ for $1 \leq k \leq N$, and takes
the value $N^{1/2}$ on each of these cubes. 
Then $\left(\sum_Q M(Q)^2\right)^{1/2} = N$. Consider the polynomial 
$$p(x) = (x_1 - 1) \dots (x_1 - N)\cdot(x_2 - 1/2N)(x_2 - 3/2N) \dots (x_2 - (2N-1)/2N)$$
which has degree $2N$, and let $Z$ denote its zero set. For a subset
$Z' \subseteq Z$ of $Z$, and for each $Q$\footnote{In this example we assume that the right-hand edge of a rectangle
belongs to the rectangle while the left-hand edge does not.}
in the support of $M$, consider the projections counted with 
multiplicities of $Z' \cap Q$, in the directions of the two standard basis vectors $e_1$ and $e_2$; let their
total lengths be $a_1(Z' \cap Q)$ and $a_2(Z' \cap Q)$ respectively, and let 
$$W(Z' \cap Q) := \{a_1(Z' \cap Q)a_2(Z' \cap Q)\}^{1/2}$$ 
be their geometric mean.  Now it transpires that the quantity $W(Z' \cap Q)$ is closely related to 
vis$(Z' \cap Q)$ -- see Section \ref{sect:vis} below -- and we shall
pretend (for the rest of this example) that $W$ 
really is the visibility. Note that if $Z_1$ and $Z_2$ are disjoint subsets of $Z \cap Q$ we 
have\footnote{The corresponding property for visibility fails.}
$$W(Z_1 \cup Z_2) \geq W(Z_1) + W(Z_2).$$ 
Now $a_1(Z \cap Q) = 1$ and $a_1(Z \cap Q) = N$ so that $W(Z \cap Q) 
= N^{1/2}$. We consider 
whether it is possible to break up each $Q$ into rectangles $R_j$ of area $1/N$
so that if $p$ bisects each rectangle then we can deduce that
$N^{1/2} =  W(Z \cap Q)$ by using
$W(Z \cap Q) \geq \sum_j W(Z \cap R_j)$. 

\smallskip
\noindent
Firstly, we could break up $Q$ into subcubes $R_j$ of side 
$N^{-1/2}$. Now for all $R_j$ except those which meet the right-hand edge of $Q$ we shall have 
$W(Z \cap R_j) = 0$, while for those which do meet the right hand edge of
$Q$ we have $W(Z \cap R_j) = N^{-1/4}$, which only gives
$W(Z \cap Q) \geq \sum_j W(Z \cap R_j) = N^{1/4}$; this is not adequate.

\smallskip
\noindent
Next, we could try breaking up $Q$ into vertical rectangles $R_j$ of sides $1/N \times 1$. Only the 
rectangle $R_0$ meeting the right-hand edge of $Q$ will have a non-zero value of $W(Z \cap R_j)$, 
and $W(Z \cap R_0) = 1$, giving $W(Z \cap Q) \geq \sum_j W(Z \cap R_j) = 1$, which is even worse.

\smallskip
\noindent
Finally, we could try breaking up $Q$ into horizontal rectangles $R_j$ of sides $1 \times 1/N$. In this case
each $W(Z \cap R_j) = N^{-1/2}$, resulting in the desired $W(Z \cap Q) \geq \sum_j W(Z \cap R_j) \geq N^{1/2}$.

\smallskip
\noindent
So only the third decomposition into horizontal rectangles is compatible with our needs. Once we have accepted 
that the polynomial $p$ above is more or less
``canonical'' for this $M$, we are essentially forced to break 
up each $Q$ into horizontal rectangles of sides $1 \times 1/N$ in order for our strategy to be successful. 
Crucially, observe that this 
decomposition reflects the global shape of the function $M$: if the support of $M$ had been along the 
$x_2$-axis we would have had to instead decompose each $Q$ into
vertical rectangles. The decomposition must therefore be aligned with the ``global
profile'' of $M$.  \hfill \qedsymbol

\bigskip
\noindent
To simplify the constructions in the proof we actually stop just short of fully
developing the moral outline given above. In fact we do not spend any time
constructing ellipsoids at the scale $\sim M(Q)^{-n}$ which would be needed
to really nail down the zero set of the polynomial we get from
the Borsuk--Ulam theorem. Instead we let ourselves be satisfied with
finding a \emph{good} polynomial $p$ with zero set $Z_p$ which satisfies
a given lower bound on the visibility
$\operatorname{vis}\,(Z_p\cap Q)$ for all cubes $Q$.
This we get by constructing ellipsoids for all \emph{bad}
polynomials,
which are those polynomials whose zero sets
have visibility less than desired on some cube.
Using these ellipsoids we show that the bad polynomials cannot cover
the unit sphere in the space of polynomials and in this way we see that
there must be a good polynomial, which gives the lower bounds mentioned
above.
Herein lies the reason for using the covering statement instead of the
Borsuk--Ulam theorem itself.
This very informal outline is developed more fully in
Section~\ref{sec:outline}.

\medskip
\noindent
For completeness, we also indicate in the following sections how Theorem \ref{main} 
follows from Theorem \ref{algtop0}, so that we give what is in essence a fully self-contained 
proof of Theorem \ref{main} (subject to the appeal to the Borsuk--Ulam theorem).
Throughout, $C$ and $c$ will denote generic constants which depend only 
on the dimension $n$ and the degree of multilinearity $d \leq n$; $P \lesssim Q$ and $P \gtrsim Q$
mean $P \leq C Q$ and  $P \geq C Q$ respectively, and $P \sim Q$ means both $P \lesssim Q$ and $P \gtrsim Q$.

\medskip
\noindent
{\em {\bf {\em Acknowledgements:}}} Both authors would like to thank Jon Bennett and Jim Wright for many illuminating 
conversations on the topics of this paper, and helpful remarks on early drafts of it. The first author would like to thank Larry Guth 
for sharing his insights into the philosophy behind the endpoint multilinear Kakeya theorem in Hyderabad in August 2010.

\section{A Preliminary reduction}
\medskip
\noindent
Recall that we have collections $ \mathcal{T}_j$, $1 \leq j \leq d$,
of $1$-tubes $T$ in $\mathbb{R}^n$ with directions $e(T) \in \mathbb{S}^{n-1}$. 
Let $\mathcal{Q}$ denote the lattice of unit cubes in $\mathbb{R}^n$.

\begin{proposition}\label{MK}
In order to prove Theorem \ref{main}, it suffices to establish the following assertion:
for every finitely supported nonnegative function $M : \mathcal{Q} \to \mathbb{R}$ satisfying
$\sum_Q M(Q)^n = 1$, there exist nonnegative functions 
$S_j : \mathcal{Q} \times \mathcal{T}_j \to \mathbb{R}$ such that for all
$T_j \in \mathcal{T}_j$ with $T_j \cap Q \neq \emptyset$,
\begin{equation}\label{need1}
e(T_1)\wedge \dots \wedge e(T_d) \, M(Q)^n \leq C S_1(Q,T_1) \dots S_d(Q,T_d)
\end{equation}
and, for all $j$ and all $T_j \in \mathcal{T}_j$ 
\begin{equation}\label{need2}
\sum_{Q \,: \, T_j \cap Q \neq \emptyset} S_j(Q, T_j) \leq C.
\end{equation}
\end{proposition}

\begin{proof}
Firstly, if we can find $S_j$ as in the statement of the proposition, homogeneity dictates that
for every finitely supported nonnegative function $M : \mathcal{Q} \to \mathbb{R}$
there exist nonnegative functions 
$S_j : \mathcal{Q} \times \mathcal{T}_j \to \mathbb{R}$ such that for all
$T_j \in \mathcal{T}_j$ with $T_j \cap Q \neq \emptyset$,
\begin{equation}\label{tyu}
e(T_1)\wedge \dots \wedge e(T_d) \, M(Q)^n \leq C S_1(Q,T_1) \dots S_d(Q,T_d) \left(\sum_Q M(Q)^n\right)^{(n-d)/n}
\end{equation}
and, for all $j$ and all $T_j \in \mathcal{T}_j$ 
\begin{equation}\label{uyt}
\sum_{Q \,: \, T_j \cap Q \neq \emptyset} S_j(Q, T_j) \leq C \left( \sum_Q M(Q)^n \right)^{1/n}.
\end{equation}

\medskip
\noindent
Secondly, we note that by the $l^1$ nature of the right hand side in
Theorem \ref{main}, we may
assume that the sets $\mathcal{T}_j$ are finite and that all the coefficients $a_{T_j}$ are equal to $1$.

\medskip
\noindent
For a unit cube $Q$ let 
$$F(Q) = \sum_{T_j \in \mathcal{T}_j \mbox{ with }T_j \cap Q \neq \emptyset} 
e(T_1)\wedge \dots \wedge e(T_d).$$
It then suffices to prove
$$ \sum_Q F(Q)^{1/(d-1)} \leq C  \left(\# \mathcal{T}_1 \dots \# \mathcal{T}_d \right)^{1/(d-1)}.$$
Let $M(Q)^n = F(Q)^{1/(d-1)} = F(Q)^{(1/(d-1) - 1/d)d}$, so that 
\begin{eqnarray*}
\begin{aligned}
&\sum_Q F(Q)^{1/(d-1)} 
= \sum_Q F(Q)^{1/d} M(Q)^{n/d} \\
= & \sum_Q \left(\sum_{T_j \in \mathcal{T}_j \mbox{ with }T_j \cap Q \neq \emptyset} 
e(T_1)\wedge \dots \wedge e(T_d)\right)^{1/d}M(Q)^{n/d} \\
= &\sum_Q \left(\sum_{T_j \in \mathcal{T}_j \mbox{ with }T_j \cap Q \neq \emptyset} 
e(T_1)\wedge \dots \wedge e(T_d) \; M(Q)^n \right)^{1/d} \\
\leq C & \sum_Q \left(\sum_{T_j \in \mathcal{T}_j \mbox{ with }T_j \cap Q \neq \emptyset} 
S_1(Q,T_1) \dots S_d(Q,T_d) \left(\sum_Q M(Q)^n\right)^{(n-d)/n} \right)^{1/d} \\
= C & \sum_Q \left(\prod_{j=1}^d \sum_{T_j \in \mathcal{T}_j \mbox{ with }T_j \cap Q \neq \emptyset} 
S_j(Q,T_j)\right)^{1/d}\left(\sum_Q M(Q)^n\right)^{(n-d)/dn} \\
\leq C & \prod_{j=1}^d \left(\sum_Q \sum_{T_j \in \mathcal{T}_j \mbox{ with }T_j \cap Q \neq \emptyset} 
S_j(Q,T_j)\right)^{1/d}\left(\sum_Q M(Q)^n\right)^{(n-d)/dn} \\
= C & \prod_{j=1}^d \left(\sum_{T_j \in \mathcal{T}_j} \sum_{Q \mbox{ with }T_j \cap Q \neq \emptyset}
S_j(Q,T_j)\right)^{1/d}\left(\sum_Q M(Q)^n\right)^{(n-d)/dn} \\
\leq C & \prod_{j=1}^d \left( \# \mathcal{T}_j\right)^{1/d} \left( \sum_Q M(Q)^n \right)^{1/n}
\left(\sum_Q M(Q)^n\right)^{(n-d)/dn}\\
= C & \prod_{j=1}^d \left( \# \mathcal{T}_j\right)^{1/d} \left( \sum_Q F(Q)^{1/(d-1)} \right)^{1/d}
\end{aligned}
\end{eqnarray*}
where the inequalities follow from \eqref{tyu}, H\"older's inequality and \eqref{uyt} respectively.
Rearranging, we obtain
$$ \left(\sum_Q F(Q)^{1/(d-1)}\right)^{(d-1)/d} 
\leq C  \prod_{j=1}^d \left( \# \mathcal{T}_j\right)^{1/d},$$
from which the result follows.
\end{proof}

\medskip
\noindent
Interestingly, the line of argument here can be reversed in certain
circumstances: assuming that the special case of the multilinear 
Kakeya theorem for transverse families of tubes $\mathcal{T}_j$
holds, it follows that for all $M$ one can 
find $S_j$ satisfying \eqref{tyu} and 
\eqref{uyt}. See \cite{CV} for more details.

\section{Directional surface area and visibility}\label{sect:vis}

\medskip
\noindent
We follow Guth \cite{G} and Bourgain--Guth \cite{BG} in defining the functions $S_j$ and 
establishing their desired properties \eqref{need1} and \eqref{need2}. In order to do this 
some geometric notions are required. We first recall the notion of directional surface area 
(termed ``directed volume'' by Guth) of a hypersurface $Z \subseteq \mathbb{R}^n$ in 
the direction of a unit vector $e$. If the element of surface area of $Z$ is denoted by 
$dS = d \mathcal{H}_{n-1}|_S$, and $e$ is a unit vector, the element of the component of 
surface area of $Z$ perpendicular to $e$ is $ |e \cdot n(x)| dS(x)$ where $n(x)$ is the 
unit normal at $x$ (which is assumed to make sense for $\mathcal{H}_{n-1}$--almost every 
$x \in Z$).
Thus the {\bf directional surface area of $Z$ in the direction $e \in \mathbb{S}^{n-1}$} is defined
as
$$ \mbox{ surf}_e(Z) = \int_Z |e \cdot n(x)|\; dS(x).$$
If $Z$ is given by the graph of a function $\Gamma : \Omega \subseteq \mathbb{R}^{n-1} \to 
\mathbb{R}$ above the hyperplane $x_n = 0$, then its directional surface area in the direction 
$e_n$ is simply the $(n-1)$-dimensional area of $\Omega$. If $Z$ is given by disjoint graphs of functions
above  the hyperplane $x_n = 0$ then its directional surface area in the direction $e_n$ is just 
$\int_{\mathbb{R}^{n-1}} J(y) \; dy$ where $J(y)$ is the number of times the line through $y$ parallel 
to $e_n$  passes through $Z$. These considerations lead immediately to Guth's ``cylinder estimate'':

\begin{lemma}[Guth's cylinder estimate]\label{cylinder}
If $T$ is a $1$-tube in $\mathbb{R}^n$ and $Z = \{x \; : \; p(x) = 0 \}$ is the zero hypersurface 
of a non-zero polynomial $p$ of degree at most $k$, then 
$$ {\rm{ surf}}_{e(T)}(Z \cap T) \leq C k.$$
\end{lemma}

\medskip
\noindent
Secondly, we associate a fundamental centrally-symmetric convex body $K(Z)$ to a hypersurface $Z$. Indeed, with $\mathbb{B}$ 
denoting the unit ball of $\mathbb{R}^n$, define
\begin{equation}\label{Kdefn}
K(Z) := \{ u \in \mathbb{B} \; : \; \mbox{ surf}_{\widehat{u}}(Z) \leq 1/|u|\}.
\end{equation}
Here $\widehat{u}$ is the unit vector in the direction of $u$. (Notice that if $Z$ is such that 
surf$_e(Z) \geq 1$ for all unit vectors $e$, then the requirement that $u$ lie in $\mathbb{B}$ is 
superfluous.) It is clear that $K(Z)$ is symmetric. To see that it is in fact convex, note that $u$ satisfies 
$\mbox{ surf}_{\widehat{u}}(Z) \leq 1/|u|$
if and only if $\int_Z | u \cdot n | \; dS \leq 1$; this condition is clearly retained under 
convex combinations of $u$'s. We then define \footnote{Guth's definition of visibility is 
the $n$'th power of the one given here, but we find the current definition more natural for 
three reasons: firstly it allows us to emphasise the ``$L^n$''-aspect of the statement of 
Theorem \ref{algtop0} which, at least when $n=d$, is no coincidence, and is a reflection of 
the fact that the optimal $L^p$ estimate for the linear Kakeya problem in $\mathbb{R}^n$ 
is conjectured to be at $p=n$; secondly it scales roughly as does $(n-1)$-dimensional Hausdorff 
measure which permits the comparison with Proposition \ref{warmup}; and thirdly, in the theory of 
finite-dimensional Banach spaces, if $K$ is a convex body in isotropic position,
the quantity $\left(\mbox{vol } K\right)^{-1/n}$ arises naturally as its isotropic constant.} the 
{\bf visibility} of $Z$ as
$$ \mbox{ vis}(Z) := \left(\mbox{vol } K(Z)\right)^{-1/n}.$$
Note that since $K(Z) \subseteq \mathbb{B}$ we always have $\mbox{ vis}(Z) \geq C$.

\medskip
\noindent
The next lemma allows us to relate visibilty to geometric means of directional surface areas.

\begin{lemma}\label{linalg}
Suppose that for all unit vectors $e \in \mathbb{R}^n$ we have $ 1 \lesssim {\rm{ surf}}_e (Z) \lesssim D$.
If $v_1, \dots , v_d,$ $( 1 \leq d\leq n )$ are unit vectors, then
$$ \left(v_1 \wedge \dots \wedge v_d\right)^{1/n} {\rm vis }(Z) 
\leq C D^{(n-d)/n} \left({\rm surf }_{v_1}(Z) \dots \, {\rm surf }_{v_d}(Z)\right)^{1/n}.$$   
\end{lemma}

\begin{proof}
We may assume that $\{v_1, \dots , v_d\}$ is linearly independent and we extend it to a 
basis $\{v_1, \dots, v_n\}$ where $v_{d+1}, \dots , v_n$ are mutually orthogonal unit vectors 
which are also orthogonal to the span of $\{v_1, \dots, v_d\}$. 

\medskip
\noindent
Since surf$_e(Z) \gtrsim 1$ for all $e$, we have that 
$\pm c v_j/\operatorname{surf}_{v_j}(Z) \in K(Z)$ for all $j$, so that by convexity of $K(Z)$
\begin{eqnarray*}
\begin{aligned}
\mbox{vol }K(Z) &\geq C \; v_1 \wedge \dots \wedge v_n \; \prod_{j=1}^n \mbox{surf }_{v_j}(Z)^{-1} \\
&=  C \; v_1 \wedge \dots \wedge v_d \; \prod_{j=1}^d \mbox{surf }_{v_j}(Z)^{-1} 
\prod_{j=d+1}^n \mbox{surf }_{v_j}(Z)^{-1} \\
&\geq   C D^{-(n-d)} \; v_1 \wedge \dots \wedge v_d \; \prod_{j=1}^d \mbox{surf }_{v_j}(Z)^{-1},
\end{aligned}
\end{eqnarray*}
from which the result follows.
\end{proof}


\medskip
\noindent
It is not hard to show that under the assumption that surf$_e(Z) \gtrsim 1$ for all $e$,
$$ \mbox{ vis}(Z) \sim \prod_{j=1}^n \mbox{ surf}_{e_j}(Z)^{1/n}$$
where $e_1, \dots, e_n$ are the principal directions of the John
ellipsoid associated to $K(Z)$ (i.e. the ellipsoid of maximal volume contained in $K(Z)$ --
see \cite{John}) and hence  
\begin{equation}\label{geomxx}
\mbox{vis}(Z) \sim \inf_{\{f_j\} \mbox{ approx orthonormal }} 
\prod_{j=1}^n \mbox{ surf}_{f_j}(Z)^{1/n}
\end{equation}
where we say that the unit vectors $f_1, \dots , 
f_n$ are  ``approximately orthonormal'' if their wedge product
satisfies $f_1 \wedge \dots \wedge f_n \geq c_n$ for a suitable
dimensional constant $c_n$. By the arithmetic-geometric mean
inequality the right-hand side of
\eqref{geomxx} is in turn dominated by $\mathcal{H}_{n-1}(Z)$. This 
shows in particular that Theorem \ref{algtop0} is morally stronger than
Proposition \ref{warmup}. 

\medskip
\noindent
The John ellipsoid $E$ of a symmetric convex body $K$ satisfies $E \subseteq K \subseteq
n^{1/2}E$, and combining the latter inclusion with Lemma \ref{cylinder} we obtain:

\begin{lemma}\label{three}
Let $p$ be a non-zero polynomial such that for some unit vector $e$,
$\operatorname{surf}_{e}(Z_{p}\cap Q) \lesssim 1.$ Then 
$$\operatorname{vis}(Z_{p} \cap Q)^{n/(n-1)} \leq C \, {\rm { deg }}\;  p.$$

\end{lemma}
\begin{proof}
    Let $E$ be the John ellipsoid associated to $K(Z_{p}\cap Q)$ and let
    $l_1\geq l_2\geq \dots \geq l_n$ be the lengths of the principal
    axes of $E$. Let $A = \operatorname{vis}(Z_{p} \cap Q)$. By
    hypothesis and the fact that $K(Z_{p}\cap Q)\subseteq n^{1/2} E$, we have $l_1\gtrsim 1$.
    Moreover we have $(l_1\dots l_n)^{-1/n}\sim A$, so $l_2\dots l_n \lesssim A^{-n}$
    and therefore $l_n\lesssim A^{-n/(n-1)}$. So if $e_n$ is the direction with which $l_n$ is associated,
    we have that $\operatorname{surf}_{e_n} (Z_{p}\cap Q)\gtrsim A^{n/(n-1)}$.
    The cylinder estimate now gives $\operatorname{deg}p \gtrsim A^{n/(n-1)}$. 
\end{proof}

\medskip
\noindent
In order to deal with a continuity issue later in the argument (in Lemma \ref{convergence} of 
Section \ref{Antipodes}), we need (as does Guth) to define variants of 
the directional surface area and visibility which are continuous functionals of 
$Z = Z_p$ when the polynomial $p$ is allowed to vary. In view of the fact that the class
of polynomials with the desired properties for Theorem \ref{algtop0} is invariant under 
multiplication by non-zero scalars, it is natural to consider the unit sphere of the class
$\mathcal{P}_k$ of polynomials of degree at most $k$ in $n$ real variables. 
Indeed, $\mathcal{P}_k$ is a vector space of dimension $\sim k^n$, and so its unit sphere 
$\mathcal{P}_k^\ast$ is homeomorphic to $\mathbb{S}^N$ where $N = N(k) \sim k^n$. So with 
$k$ fixed, we allow $p$ to vary within  $\mathcal{P}_k^\ast$.\footnote{From now on we shall use 
the notations $\mathcal{P}_k^\ast$ and $\mathbb{S}^N$ where $N= N(k)$ interchangeably, the former 
when we are thinking of individual polynomials, the latter when continuity and topological 
considerations are foremost.} 
The continuity property needed
is most simply achieved by replacing $\mbox{surf}_e(Z)$ for surfaces of the form
$Z=Z_p\cap U$ (where $p\in \mathbb{S}^N$ and $U$ is open in $\mathbb{R}^n$) by
$\operatorname{surf}_{e,\varepsilon}(Z)$ which we define as the average of
$\mbox{surf}_e(Z')$ with $Z'=Z_{p'}\cap U$ over $p'$
in a ball of radius $\varepsilon$ centred at $p$ in $\mathbb{S}^N$.
From this we define $K_\varepsilon(Z)$ and
$\operatorname{vis}_\varepsilon(Z)$ in analogy to $K(Z)$ and $\operatorname{vis}(Z)$.
In the argument we will have to choose $\varepsilon$ sufficiently small so that these
entities behave in certain ways similarly to the unmollified versions.

\medskip
\noindent
It is a routine matter to verify that $K_\varepsilon(Z)$ is convex and that the
three lemmas of this section hold with these mollified variants. To be precise, 
fixing $k$ and the associated definitions of $\operatorname{surf}_{e,\varepsilon}(Z)$,
$K_\varepsilon(Z)$ and $\operatorname{vis}_\varepsilon(Z)$ as above
for $\mathcal{P}_k^\ast$, we have, (with implicit constants independent of $\varepsilon > 0$):

\begin{lemma}\label{cylinder-moll} 
If $T$ is a $1$-tube in $\mathbb{R}^n$ and $Z = \{x \; : \; p(x) = 0 \}$ is the zero hypersurface 
of a polynomial $ p \in \mathcal{P}_k^\ast$, then for all $\varepsilon>0$,
$$ {\rm{ surf}}_{e(T), \varepsilon}(Z \cap T) \leq C k.$$
\end{lemma}

\begin{lemma}\label{linalg-moll}
Suppose that $p \in \mathcal{P}_k^\ast$ and that $Z=Z_p\cap U$ as above. 
Also suppose that for some $\varepsilon>0$ and
all unit vectors $e \in \mathbb{R}^n$ we have 
$ 1 \lesssim {\rm{ surf}}_{e, \varepsilon} (Z) \lesssim D$.
If $v_1, \dots , v_d,$ $( 1 \leq d\leq n )$ are unit vectors, then,
$$ \left(v_1 \wedge \dots \wedge v_d\right)^{1/n} {\rm vis_{\varepsilon} }(Z) 
\leq C D^{(n-d)/n} \left({\rm surf }_{v_1, \varepsilon}(Z) \dots \, {\rm surf }_{v_d, \varepsilon}(Z)\right)^{1/n}.$$
\end{lemma}

\begin{lemma}\label{three-moll}
Suppose $p \in \mathcal{P}_k^\ast$ is such that for some $\varepsilon >0$ and some unit vector $e$, 
$\operatorname{surf}_{e, \varepsilon}(Z_{p}\cap Q) \lesssim 1.$ Then 
$$\operatorname{vis}_\varepsilon(Z_{p} \cap Q)^{n/(n-1)}\leq Ck.$$
\end{lemma}

\medskip
\noindent
The reader may wish to proceed with the unmollified variants in mind 
on a first reading.

\section{Application of the main result to multilinear Kakeya}
\noindent
The version of Theorem \ref{algtop0} that we will actually need is:

\begin{theorem}\label{algtop}
Given a nonnegative function $M : \mathcal{Q} \to \mathbb{R}$,
there exists a nonnegative integer $k$, a polynomial $p \in \mathcal{P}_k^\ast$ and an $\varepsilon>0$ 
such that 
$$ k \leq C \left(\sum_Q M(Q)^n\right)^{1/n}$$
and such that for all $Q \in \mathcal{Q}$ 
$$ \operatorname{vis}_\varepsilon(Z_p \cap Q) \geq C M(Q).$$
\end{theorem}
\noindent
(Note that since it is always the case that $\operatorname{vis}_\varepsilon(Z_p \cap Q) \gtrsim 1$, the 
latter condition has content only for those $Q$ with $M(Q) \gtrsim 1$.)

\medskip
\noindent
In this section we show how this result implies the conditions of Proposition \ref{MK}, that is, 
given a finitely supported nonnegative function $M : \mathcal{Q} \to \mathbb{R}$ satisfying
$\sum_Q M(Q)^n = 1$, there exist nonnegative functions 
$S_j : \mathcal{Q} \times \mathcal{T}_j \to \mathbb{R}$ such that \eqref{need1} and \eqref{need2} hold.


\medskip
\noindent
Given such a finitely supported nonnegative function $M(Q)$ with $\sum_Q M(Q)^n = 1$, we define $M_0(Q) = \lambda M(Q)$ for 
some $\lambda \gg 1$ which is required to satisfy $\lambda \geq M(Q)^{-n}$ for all $Q$ in the support of $M$. 
Apply Theorem \ref{algtop} with data $M_0$ to obtain a $k$, a $p \in \mathcal{P}_k^\ast$ and an 
$\varepsilon > 0$ such that
\begin{equation}
    \label{deg} k \leq C \lambda
\end{equation}
and
\begin{equation}
    \label{wer}
    \operatorname{vis}_\varepsilon(Z_{p} \cap Q) \geq C \lambda M(Q).
\end{equation}


\medskip
\noindent
Using \eqref{deg}, \eqref{wer} and our requirement on $\lambda$, we have
\begin{eqnarray*}
\begin{aligned}
k &\leq C \lambda = C \lambda^{n/(n-1)} \lambda^{-1/(n-1)} \\
&\leq C \left(\frac{\operatorname{vis}_\varepsilon(Z_{p} \cap Q)}{M(Q)}\right)^{n/(n-1)}
\left(M(Q)^{-n}\right)^{-1/(n-1)}\\
&= C \operatorname{vis}_\varepsilon(Z_{p} \cap Q)^{n/(n-1)}.
\end{aligned}
\end{eqnarray*}

\medskip
\noindent
Using Lemma~\ref{three-moll} we deduce that for all cubes $Q$ in the support of $M$ and all unit vectors $e$ we have 
$\operatorname{surf}_{e,\varepsilon}(Z_{p}\cap Q) \gtrsim 1$.\footnote{Another way of achieving this is to 
multiply the polynomial which Theorem~\ref{algtop}
produces with a polynomial whose zero set consists of
hyperplanes parallel to the coordinate hyperplanes which pass through the cubes in the
support of $M$. This has an insignificant effect on the degree of the polynomial provided $\lambda$ is large enough.
However, some care must be taken when considering how this augmentation interacts with the mollification.} This in turn will permit us to apply Lemma~\ref{linalg-moll} (relating visibility to geometric means of directional surface areas) below.

\medskip
\noindent
We turn to the verification of \eqref{need2}. By Lemma~\ref{cylinder-moll} and \eqref{deg} we have, for all $e \in \mathbb{S}^{n-1}$,
\begin{equation}\label{awq}
{\rm surf }_{e,\varepsilon}(Z_p \cap Q) \leq C \lambda
\end{equation} 
and moreover
\begin{equation}\label{qwa}
\sum_{Q \, : \, Q \cap T_j \neq \emptyset} \mbox{ surf}_{e(T_j),\varepsilon}({Z_p} \cap Q)
\leq C \mbox{ surf}_{e(T_j),\varepsilon}({Z_p} \cap \tilde{T}_j) 
\leq C k \leq C \lambda , 
\end{equation}
(where $\tilde{T}$ denotes the expansion of a tube $T$ about its axis by a dimensional factor).

\medskip
\noindent
We now define 
$$S_j(Q, T_j) : = \lambda^{-1} \mbox{ surf}_{e(T_j),\varepsilon}({Z_p} \cap Q)$$ 
and observe that \eqref{qwa} immediately implies
$$ \sum_{Q \, : \, Q \cap T_j \neq \emptyset} S_j(Q, T_j) \leq C, $$
which establishes \eqref{need2}.

\medskip
\noindent
On the other hand, to see that \eqref{need1} is satisfied, note that 
\eqref{wer} and \eqref{awq} together with Lemma \ref{linalg-moll} give
$$ C \lambda M(Q) \leq {\rm vis_\varepsilon }(Z_p \cap Q) \leq 
\frac{C  \lambda^{(n-d)/n} \left({\rm surf }_{e(T_1),\varepsilon}(Z_p \cap Q) 
\dots \, {\rm surf }_{e(T_d),\varepsilon}(Z_p \cap Q)\right)^{1/n}}
{\left(e(T_1) \wedge \dots \wedge e(T_d)\right)^{1/n}},$$
and so

$$ S_1(Q,T_1) \dots S_d(Q,T_d) = \lambda^{-d}\prod_{j=1}^d \mbox{ surf}_{e(T_j),\varepsilon}({Z_p} \cap Q)$$
is at least
$$ C \lambda^{-d}\lambda^n M(Q)^n \lambda^{d-n} e(T_1)\wedge \dots \wedge e(T_d) 
= C M(Q)^n e(T_1)\wedge \dots \wedge e(T_d),$$
and thus \eqref{need1} is established.

\medskip
\noindent
Consequently, the multilinear Kakeya theorem is reduced to proving Theorem \ref{algtop}.

\section{The Borsuk--Ulam theorem and a covering lemma}\label{sec:BU}

\noindent
The Borsuk--Ulam theorem is as follows:

\begin{theorem}[Borsuk--Ulam]
Suppose that $N \geq J$ and that $F: \mathbb{S}^N \to \mathbb{R}^J$ is continuous and satisfies 
$F(-x) = - F(x)$ for all $x \in \mathbb{S}^N$. Then there is an $x \in \mathbb{S}^N$ such that 
$F(x) = 0$.
\end{theorem}
\noindent

\medskip
\noindent
For a delightful discussion of this theorem and its applications, see \cite{Mat}. See also \cite{C}
for a recent proof of the Borsuk--Ulam theorem using only point-set topology and Stokes' theorem. 
Included in \cite{Mat} there is a discussion of various equivalent forms of this theorem, some of which 
(known as Lusternik--Schnirelmann results) take the form of covering statements for 
the sphere. In this section we formulate another such equivalent covering statement which we shall use 
in our proof of Theorem \ref{algtop}.

\begin{lemma}\label{covering}
Suppose that $A_i \subseteq \mathbb{S}^N$ for $ 1 \leq i \leq J$, and suppose that for each $i$,
$A_i \cap (\overline{-A_i}) = \emptyset$. 
If $J \leq N$, then the $2J$ sets $A_i$ and $-A_i$ do not cover $\mathbb{S}^N$. 
\end{lemma}

\medskip
\noindent
Note that no topological hypothesis on the sets $A_i$ is needed.
 
\begin{proof}
Let $F_i(x) = d(x, -A_i) - d(x, A_i)$ for $1 \leq i \leq J$. Then, with $F: \mathbb{S}^N \to \mathbb{R}^J$
defined by $F(x) = (F_1(x), \dots , F_J(x))$, we have  that $F$ is continuous and $F(-x) = -F(x)$ for all $x$, 
so by the Borsuk--Ulam theorem there is an $x$ with $F(x) = 0$. We claim that this $x$ does not belong 
to any $A_i$ or to $-A_i$. For if $x \in A_i$ we have
$d(x, A_i) = 0$, hence $d(x, -A_i) = 0$, hence $x \in
\overline{-A_i}$, a contradiction. Likewise, since by hypothesis $
\overline{A_i} \cap (-A_i) = \emptyset$, 
$x \in -A_i$ gives  $x \in \overline{A_i}$, another contradiction.
\end{proof}

\medskip
\noindent
{\bf{Remark.}} The converse argument also holds: if we assume the assertion of the lemma, 
{\em{but only for open sets}} $U_i$, we can recover the Borsuk--Ulam theorem. Indeed, suppose 
$F: \mathbb{S}^N \rightarrow \mathbb{R}^J$ is continuous, $F(-x) = -F(x)$ and $N \geq J$. Let
$U_i$ be the open set $\{x \; : \; F_i(x) > 0\}$. Then $-U_i = \{x \; : \; F_i(x) < 0\}$, so that
$\overline{- U_i} \subseteq \{x \, : \, F_i(x) \leq 0\}$ and so
$U_i \cap (\overline{-U_i}) = \emptyset$.
By assumption there is an $x$ which is not in any of the $U_i$ or $-U_i$. So $F_i(x) \leq 0$ for 
all $i$ and $F_i(x) \geq 0$ for all $i$. Hence $F_i(x) = 0$ for all $i$, that is, $F(x) = 0$.\footnote{That
the assertion of Lemma \ref{covering} for open sets $U_i$ logically implies the same 
statement for sets $A_i$ with no topological restrictions can easily be seen directly from the fact that the 
metric space $\mathbb{S}^N$ satisfies appropriate separation axioms. More precisely, if two subsets $A$ and $B$ 
of a metric space are separated in the sense that the closure of either does not meet the other, then there are 
open sets $U$ and $V$ with  the same property such that $A \subseteq U$ and $B \subseteq V$.}

\section{Outline of the proof of Theorem \ref{algtop}} 
\label{sec:outline}

\noindent
We now describe the scheme of the proof of Theorem \ref{algtop}. The function $M$ is given, 
and we will be working with the class $\mathcal{P}_k^\ast = \mathbb{S}^N$
of normalised polynomials 
$p : \mathbb{R}^n \to \mathbb{R}$ of degree bounded by some $k \in \mathbb{N}$. Recall that $N \sim k^n$.
For each such polynomial $p$, its zero set is the algebraic hypersurface 
$Z_p = \{x \, : \, p(x) = 0 \}$, 
and we let 
$$S(Q) = \{ p \in \mathcal{P}_k^\ast \; : \; \operatorname{vis}_\varepsilon(Z_p \cap Q) \leq M(Q)\}.$$  
Following Guth \cite{G}, the aim is to show that if we take a suitable 
$k \sim (\sum_Q M(Q)^n)^{1/n}$, and a suitable $\varepsilon >0$, then we can find a 
polynomial in $\mathcal{P}_k^\ast$ which is not 
in any of the $S(Q)$. (Note that $S(Q) = \emptyset$ 
for those $Q$ such that $M(Q) \lesssim 1$.) Our method to 
establish this diverges somewhat from that of Guth, but there are of course many points of contact 
between the two lines of argument.

\medskip
\noindent
Let, for $r \geq 0$,
$$S^{(r)}(Q) = \{ p \in \mathcal{P}_k^\ast \; : \; \operatorname{vis}_\varepsilon(Z_p \cap Q) \sim 2^{-r} M(Q)\}.$$ 
Then 
$$ S(Q) = \bigcup_{1 \lesssim 2^r \lesssim M(Q)}S^{(r)}(Q)$$ 
since $S^{(r)}(Q) = \emptyset$ for $r$ such that $2^r \gtrsim M(Q)$.
 
\medskip
\noindent
We shall introduce a collection $\mathcal{C}$ of ``colours'' $\Theta$ whose cardinality is bounded by $C$. 
For each colour $\Theta$ we shall define subsets $S^{(r), \Theta}(Q)$ of $S^{(r)}(Q)$
which have the property that
\begin{equation}\label{colours}
S^{(r)}(Q) = \bigcup_{\Theta \in \mathcal{C}} S^{(r), \Theta}(Q).
\end{equation}

\medskip
\noindent
For each fixed $Q$ and $r$ such that $1 \lesssim 2^r \lesssim M(Q)$ we will 
define an indexing set $\mathcal{A}_{Q, r}$ of cardinality
$C 2^{-rn} M(Q)^n$, and for each $\alpha \in \mathcal{A}_{Q, r}$, 
a subset $S^{(r), \Theta}_\alpha (Q)$ of $S^{(r), \Theta}(Q)$ such that 
\begin{equation}\label{translates}
S^{(r), \Theta}(Q) = \bigcup_{\alpha \in \mathcal{A}_{Q, r}}
S^{(r), \Theta}_\alpha (Q).
\end{equation}
To ensure that this decomposition is well-defined, it will transpire that $\varepsilon$ must be taken to be small.


\medskip
\noindent
Finally we shall decompose each $S^{(r), \Theta}_\alpha (Q)$ as
$$ S^{(r), \Theta}_\alpha (Q) = S^{(r), \Theta + }_\alpha (Q) \cup S^{(r), \Theta - }_\alpha (Q),$$
where 
\begin{equation}\label{opposite}
S^{(r), \Theta -}_\alpha(Q) = -  S^{(r), \Theta +}_\alpha(Q)
\end{equation} 
in such a way that for all $Q, r, \Theta$ and $\alpha$,
\begin{equation}\label{qw}
S^{(r), \Theta + }_\alpha (Q) \cap \overline{S^{(r), \Theta - }_\alpha (Q)} = \emptyset.
\end{equation}

\medskip
\noindent
(The closure here refers to the natural topology of $\mathcal{P}_k^\ast = \mathbb{S}^N$.)

\medskip
\noindent
The reason for the introduction of colours is to ensure that
there is sufficient separation between the sets $S^{(r), \Theta \pm
}_\alpha (Q)$ and their antipodes for \eqref{qw} to hold.

\medskip
\noindent
In summary then, 
\begin{equation}\label{union}
\bigcup_Q S(Q) = \bigcup_Q \bigcup_{1 \lesssim 2^r \lesssim M(Q)}\bigcup_{\Theta \in \mathcal{C}} 
\bigcup_{\alpha \in \mathcal{A}_{Q, r}} \left(S^{(r), \Theta + }_\alpha (Q) 
\cup S^{(r), \Theta - }_\alpha (Q)\right),
\end{equation}
where $S^{(r), \Theta + }_\alpha (Q)$ and $S^{(r), \Theta - }_\alpha (Q)$
satisfy \eqref{opposite} and \eqref{qw}. 

\medskip
\noindent
Lemma \ref{covering} then implies that if the cardinality of the set indexing the union on the 
right hand side of \eqref{union} is less than or equal to $N$,
then the sets in the union cannot cover $\mathbb{S}^N = \mathcal{P}_k^\ast$. 

\medskip
\noindent
Now the number of terms indexing the union is at most
$$ C \sum_Q \sum_{r \geq 0} \sum_{\Theta \in \mathcal{C}} \, \sum_{\alpha \in \mathcal{A}_{Q, r}} 
1 \leq C \sum_Q \sum_{r \geq 0} 2^{-rn} M(Q)^n \leq C \sum_Q M(Q)^n.$$
So provided that $N \gtrsim \sum_Q M(Q)^n$, amongst the polynomials in $\mathbb{S}^N =
\mathcal{P}_k^\ast$, there will exist one which is not in any of the $S(Q)$. Since 
$N \sim k^n$, we can therefore take $k$ with $k \sim \left(\sum_Q M(Q)^n \right)^{1/n}$
and a $p \in \mathcal{P}_k^\ast$ which, for suitable $\varepsilon > 0$, satisfies 
$\operatorname{vis}_\varepsilon(Z_p \cap Q) > M(Q)$ for all $Q$, as was needed.

\medskip
\noindent
It remains now to define the various decompositions introduced above,
and establish the assertions we have made concerning them.

\section{Colours}
\noindent
In this section we describe how to establish \eqref{colours} in such a way that the indexing set 
$\mathcal{C}$ has cardinality at most $C$.

\medskip
\noindent
Let $\mathcal{E}$ denote the class of centred ellipsoids in $\mathbb{R}^n$, 
that is images of the unit ball $\mathbb{B}$ 
by affine linear maps $A$. Each ellipsoid $A (\mathbb{B})$ is determined by an orthonormal basis of 
principal axes or directions given by the normalised eigenvectors of $A^tA$, and corresponding semiaxes, 
the squares of whose lengths are the eigenvalues of $A^tA$. 
Thus $\mathcal{E}$ is a manifold of dimension
$n(n-1)/2 + n = n(n+1)/2$. 

\medskip
\noindent
Let $\mathcal{K}$ denote the class of centrally symmetric convex bodies in $\mathbb{R}^n$.
By the John ellipsoid theorem \cite{John}, every member $K$ of $\mathcal{K}$ is close to some ellipsoid $E$ 
in the sense that $n^{-1/4} E \subseteq K \subseteq n^{1/4} E$. 

\medskip
\noindent
There is a natural metric (the Banach--Mazur metric) to put on the class $\mathcal{K}$, given by
$$d(K,L) = \log \inf \{ \alpha \geq 1 \, : \, \alpha^{-1}K \subseteq L \subseteq \alpha K\}.$$    
The John ellipsoid theorem asserts that every convex body is distant at most 
$(\log n)/4 \lesssim 1$  from some ellipsoid. An ellipsoid with semiaxes of lengths $2^{k_1}, \dots , 2^{k_n}$ where
$k_1 + \dots + k_n =0$ will be distant $ \lesssim \max |k_j| $ from the unit ball. Two congruent ellipses 
in $\mathbb{R}^2$ with semiaxes of lengths $1$ and $N$ and principal directions differing by $\theta$ will be 
distant $\lesssim \theta N $ apart.

\medskip
\noindent
To set the scene for the covering property of ellipsoids that we need, note that in $\mathbb{R}^N$, 
given a scale $\rho > 0$ and a pre-assigned number $\gamma > 1$, 
we can find a set $\mathcal{X}$ of $\rho$-separated points $x_i \in \mathbb{R}^N$ such that every 
point of $\mathbb{R}^N$ is in some $B(x_i,\rho)$, and such that $\mathcal{X}$ can be partitioned 
into $O_N(\gamma^N)$ families (colours), so that points of $\mathcal{X}$ of the same colour are 
distant at least $ \gamma \rho$ from each other. This property expresses the idea that the 
dimensionality of $\mathbb{R}^N$ as a metric space is $N$. We can then assign to each $x \in \mathbb{R}^N$ 
one or more colours according to whether $d(x, x_i) < \rho$ for some $x_i \in \mathcal{X}$ of that 
particular colour. 

\medskip
\noindent
Similarly it is not hard to verify that given $\rho >0$ and $\gamma >1$, there exists a $\rho$-separated 
subset $\mathcal{E}_0$ of $\mathcal{E}$ such that $\{B(E, \rho) \, : \, E \in  \mathcal{E}_0\}$ 
covers $\mathcal{E}$ and such that we can partition $\mathcal{E}_0$ into at most 
$O_n( \gamma^{n(n+1)/2})$ families (colours) such that any two ellipsoids in $\mathcal{E}_0$ 
of the same colour are distant at least $\gamma \rho$ from each other.  

\medskip
\noindent
Choosing $\rho = 1$ and $\gamma$ sufficiently large depending only on the dimension $n$, and using the 
John ellipsoid theorem, we obtain the following: 

\begin{lemma}\label{ellipsoid}
Supose $\alpha_n > 1 $ is sufficiently large. 
Then there exists a set $\mathcal{E}_0 \subseteq \mathcal{E}$ with the property that for every 
$K \in \mathcal{K}$ there is an $E \in \mathcal{E}_0$ such that
\begin{equation}\label{close}
\alpha_n^{-1} K \subseteq E \subseteq \alpha_n  K
\end{equation}
and such that the set $\mathcal{E}_0$ can be partitioned into at most $C = C(n, \alpha_n)$
colours in such a way that every $K \in \mathcal{K}$ satisfies \eqref{close} for at most one 
$E \in \mathcal{E}_0$ {\em of a given colour}.
\end{lemma}


\medskip
\noindent
Given $n$ we now fix $\alpha_n$ sufficiently large, and fix our palette $\mathcal{C}$ 
consisting of at most $C(n, \alpha_n)$ colours once and for all so that the conclusion of 
Lemma \ref{ellipsoid} holds. We say that two convex bodies $E$ and $K$ are {\bf close} if 
\eqref{close} holds. So every $K \in \mathcal{K}$ is close to some member of $\mathcal{E}_0$, 
but there is at most one $E \in \mathcal{E}_0$ of a given colour to which it is close.  
For a colour $\Theta \in \mathcal{C}$ let 
$$\mathcal{E}^{\Theta}_0 = \{ E \in \mathcal{E}_0 \; : \; E \mbox{ is of colour  } \Theta\}.$$

\medskip
\noindent
Finally, given $Q$ and $r \geq 0$, 
let
$$S^{(r), \Theta} (Q) = \{ p \in S^{(r)}(Q) \, : \, K_\varepsilon(Z_p \cap Q) {\mbox{ is close to a member of }}
\mathcal{E}^{\Theta}_0\};$$
then we have
$$S^{(r)}(Q) = \bigcup_{\Theta \in \mathcal{C}} S^{(r), \Theta}(Q),$$
and \eqref{colours} is established.


\section{Translates}\label{sec:translates}
\noindent
We now fix $Q$, $r \geq 0$ and a colour $\Theta \in \mathcal{C}$. In this section we establish 
\eqref{translates} for suitable subsets $S^{(r), \Theta}_\alpha(Q) \subseteq 
S^{(r), \Theta}(Q)$ which are indexed by $\alpha \in \mathcal{A}_{Q, r}$, where $\mathcal{A}_{Q, r}$
has cardinality $\sim 2^{-rn} M(Q)^n$. We can assume that  $S^{(r), \Theta}(Q) \neq \emptyset$.

\medskip
\noindent
If $p \in S^{(r), \Theta}(Q)$, the convex body $K_\varepsilon(Z_p \cap Q) \subseteq \mathbb{B}$ 
has volume $\sim 2^{rn}/M(Q)^{n} \lesssim 1$, and it is close to a unique member $E(p)$ of $\mathcal{E}^{\Theta}_0$ 
of comparable volume. Hence we can fit $\sim 2^{-rn} M(Q)^{n}$ disjoint parallel translates of $E(p)$
inside $Q$, with the translations along the principal directions of $E(p)$. Likewise,
if $\eta < 1$ is a numerical scaling factor, we can fit $\sim \eta^{-n}2^{-rn} M(Q)^{n}$ disjoint parallel 
translates of $\eta E(p)$ inside $Q$, with the translations again along the principal directions of $E(p)$.
Indeed, if the lengths of the semiaxes of $E(p)$ are $l_1, \dots, l_n
\leq c$, and the principal directions are 
$e_1, \dots, e_n$, we can place the centres of the translated copies of $\eta E(p)$ at the points  
$x_Q + \eta \sum_j m_j l_j e_j$ for $m_j \in 2 \mathbb{Z}$ and $|m_j| \leq c \eta^{-1}l_j^{-1}$; 
here $x_Q$ is the centre of $Q$.
In this construction
the number of translated copies equals the product
\begin{equation}
    \label{numtranslates}
    c\eta^{-n}(l_1\dots l_n)^{-1}=c \eta^{-n}2^{-rn} M(Q)^{n}.
\end{equation}

\begin{lemma}
    \label{manybisections}
There is a dimensional constant $C_n$ such that if $p \in S^{(r), \Theta}(Q)$
and $\eta < 1$, then $Z_p$ bisects 
at most $ C_n \eta ^{-(n-1)} 2^{-rn} M(Q)^{n}$ disjoint translates of $\eta E(p)$ in $Q$.
\end{lemma}

\begin{proof}
Suppose that $E(p)$ has principal directions $\{e_j\}$ and corresponding semiaxes with lengths $\{l_j\}$. 
If $Z_p$ bisects a translate $\eta E(p) + \xi$ of $\eta E(p)$, then for at least one $j$ we will 
have\footnote{Here and in the next two displayed equations we are using the {\em unmollified} 
notion of directional surface area $\mbox{surf}_e$.}
$$ \mbox{ surf}_{e_j}(Z_p \cap (\eta E(p) + \xi)) \geq C_n \operatorname{vol}(\eta E(p) + \xi) / \eta l_j
= C_n \eta ^{n-1} \operatorname{vol}E(p)/l_j.$$
This is just the affine-invariant formulation of the fact that a hypersurface which bisects the 
unit ball must have large $(n-1)$-dimensional Hausdorff measure inside
the ball -- see Lemma \ref{bisect} in the Appendix.
So 
$$\sum_{j=1}^n l_j \operatorname{surf}_{e_j}(Z_p \cap (\eta E(p) + \xi))
 \geq C_n \eta ^{n-1} \operatorname{vol}E(p).$$
If now $Z_p$ bisects as many as $ A 2^{-rn} M(Q)^{n}$ disjoint copies of $\eta E(p)$ in $Q$, 
we will have
$$\sum_{j=1}^n l_j \operatorname{surf}_{e_j}(Z_p \cap Q) \geq 
C_n A \eta^{n-1} \operatorname{vol}E(p)\,\, 2^{-rn} M(Q)^{n}.$$

\medskip
\noindent
If $p'$ is another polynomial in $\mathbb{S}^N$ sufficiently close to $p$,
and if $p$ bisects an ellispoid $E$, we can conclude that $p'$ approximately bisects $E$ in the sense that
$p'$ is positive on at least $40\%$ of $E$ and negative on at least $40\%$ of $E$.
Since we are only interested in a finite number of ellipsoids here, namely the translates $\eta E+\xi$
for $E\in\mathcal{E}_0$ with $M(Q)^{-n}\lesssim\operatorname{vol}E\lesssim 1$ for the relevant $Q$,
then by choosing $\varepsilon$ small enough we will have this approximate bisection property for all
polynomials which affect the value of $\operatorname{surf}_{e,\varepsilon}$ in the expressions above.
Therefore we have the estimate
$$\sum_{j=1}^n l_j \operatorname{surf}_{e_j,\varepsilon}(Z_p \cap Q) \geq 
C_n A \eta^{n-1} \operatorname{vol}E(p)\,\, 2^{-rn} M(Q)^{n}.$$

\medskip
\noindent
The definition (cf. \eqref{Kdefn}) of $K_\varepsilon(Z_p \cap Q)$ together with the fact that
$K_\varepsilon(Z_p \cap Q)$ and $E(p)$ are close 
implies that 
$$ l_j  \operatorname{surf}_{e_j,\varepsilon}(Z_p \cap Q) \leq C_n$$
for each $j$, and moreover, as $p \in  S^{(r), \Theta}(Q)$,
$$\operatorname{vol}E(p) \sim 2^{rn}/M(Q)^{n}.$$ 
Therefore $A$ must satisfy $A  \leq C_n \eta^{-(n-1)}$.
\end{proof}

\medskip
\noindent
Choose $\eta$ sufficiently small so that $C_n \eta <c$.
Here $C_n$ is the constant from Lemma~\ref{manybisections}
and $c$ is the constant in~\eqref{numtranslates}.
Then
for each $p \in S^{(r), \Theta}(Q)$, $Z_p$ can bisect only a 
proportion $C_n \eta/c < 1$ of the $c\eta ^{-n} 2^{-rn} M(Q)^{n}$ disjoint
copies of $\eta E(p)$ in $Q$. In particular, $Z_p$ will {\em not} bisect
{\em all} of the available disjoint copies of $\eta E(p)$ in $Q$.

\medskip
\noindent
For each $E(p) \in \mathcal{E}^{\Theta}_0$ of volume $\sim 2^{rn}/M(Q)^{n}$,
the set of translated ellipsoids which are placed into $Q$ in the construction above
is of cardinality $c2^{-rn} M(Q)^{n}$.
We label these ellipsoids with an index $\alpha$ from $\mathcal{A}_{Q,r}=\{1,\dots, c2^{-rn} M(Q)^{n}\}$.
Now take a polynomial $p\in S^{(r), \Theta}_\alpha (Q)$.
Then $p$ is close to a unique member $E(p)\in \mathcal{E}_0^\Theta$ of volume comparable to
$2^{rn}/M(Q)^{n}$ and we can ask whether or not $p$ bisects the $\alpha$'th translate
of $\eta E(p)$. Note that for this question to be meaningful the uniqueness in
the previous sentence is important and if we had not already restricted attention to a single colour then
the uniqueness would not hold.

\medskip
\noindent
For $\alpha \in \mathcal{A}_{Q,r}$ we can thus define
\begin{eqnarray*}
S^{(r), \Theta}_\alpha (Q) 
:= S^{(r), \Theta}(Q) \cap \{Z_p \mbox{ does not bisect the $\alpha$'th translate of 
$\eta E(p)$ in $Q$}\} 
\end{eqnarray*}


\medskip
\noindent
Then, since $Z_p$ cannot bisect all of the translates of $\eta  E(p)$ in $Q$, we have 
$$S^{(r), \Theta}(Q) =  
\bigcup_{\alpha \in \mathcal{A}_{Q, r}} S^{(r), \Theta}_\alpha (Q)$$
as required, where there are $c 2^{-rn} M(Q)^{n}$ terms in the indexing set 
$\mathcal{A}_{Q, r}$. Thus \eqref{translates} is established.

\section{Antipodes}\label{Antipodes}

\medskip
\noindent
In this section we establish \eqref{qw}.

\medskip
\noindent
Fix $Q$, $r \geq 0$, $\Theta \in \mathcal{C}$ and $\alpha \in \mathcal{A}_{Q, r}$. 
For $p \in S^{(r), \Theta}_\alpha (Q)$ let $E(p)$ be as before the unique member of $\mathcal{E}_0^\Theta$ 
to which $K_\varepsilon(Z_p \cap Q)$ is close. Let
$E_\alpha = E_\alpha(p)$ denote the $\alpha$'th translate of $\eta E(p)$ in $Q$.
Since $Z_p$ does not bisect $E_\alpha$,
we have either 
$$ \mbox{ vol }(\{p>0\} \; \cap \; E_\alpha) > \mbox{ vol}(\{p<0 \}
\; \cap \; E_\alpha) $$
(in which case we say that $p \in S^{(r), \Theta +}_\alpha(Q)$), or 
$$ \mbox{ vol }(\{p>0\} \; \cap \; E_\alpha) < \mbox{ vol}(\{p<0 \}
\; \cap \; E_\alpha), $$
(in which case we say that $p \in S^{(r), \Theta -}_\alpha(Q)$).

\medskip
\noindent
Then
$$ S^{(r), \Theta}_\alpha(Q) = S^{(r), \Theta +}_\alpha(Q) \cup S^{(r), \Theta -}_\alpha(Q).$$ 
Moreover 
$$ S^{(r), \Theta -}_\alpha(Q) = -  S^{(r), \Theta +}_\alpha(Q),$$ 
and so to establish \eqref{qw} we wish to show that for all $\alpha$,  
$$S^{(r), \Theta +}_
\alpha(Q) \cap \overline{S^{(r), \Theta -}_\alpha(Q)} = \emptyset.
$$

\medskip
\noindent
To see this, suppose for a contradiction that for some
$\alpha \in \mathcal{A}_{Q, r}$ there is a $p \in S^{(r),
\Theta +}_\alpha(Q)$ and a sequence of $p_m \in  S^{(r), \Theta -}_\alpha(Q)$
which converges to $p$ in $\mathbb{S}^N$. That is, we suppose that 
\begin{equation}\label{last}
\mbox{ vol }(\{p>0\} \; \cap \; E_\alpha (p)) > \mbox{ vol}(\{p<0 \}
\; \cap \; E_\alpha(p)) 
\end{equation}
and
$$ \mbox{ vol }(\{p_m> 0\} \; \cap \; E_\alpha (p_m)) < \mbox{ vol}(\{p_m <0 \}
\; \cap \; E_\alpha(p_m))$$
where  $p_m$ converges to $p$ in $\mathbb{S}^N$.

\begin{lemma}\label{convergence}
Fix $Q$, $r$ and $\Theta$. Suppose that $p \in S^{(r), \Theta}(Q)$, $p_m \in S^{(r), \Theta}(Q)$ 
for $m \in \mathbb{N}$ and that $p_m$ converges to $p$ in $\mathbb{S}^N$. Then for all sufficiently 
large $m$ we have $E(p_m) = E(p)$. If $\alpha \in \mathcal{A}_{Q,r}$ and in addition 
$p, p_m \in S^{(r), \Theta}_\alpha(Q)$, then for $m$ sufficiently large, 
$E_\alpha(p_m) = E_\alpha(p)$.
\end{lemma}

\begin{proof}
Since we are using the mollified version of the directional surface area and quantities defined in terms of it, 
the convergence of $p_m$ to $p$ in $\mathbb{S}^N$ implies that the convex bodies $K_\varepsilon(Z_{p_m} \cap Q)$
converge to $K_\varepsilon(Z_{p} \cap Q)$ as $m \to \infty$\footnote{in the sense that there is a sequence $\gamma_m \geq 1$ 
with $\gamma_m \to 1$ such that
$\gamma_m^{-1} K_\varepsilon (Z_{p_m} \cap Q) \subseteq K_\varepsilon(Z_{p} \cap Q) \subseteq \gamma_m 
K_\varepsilon(Z_{p_m} \cap Q)$} and in particular $K_\varepsilon(Z_{p_m} \cap Q)$ and $K_\varepsilon(Z_{p} \cap Q)$
are close for $m$ sufficiently large.
Since $p$ and $p_m$ are members of $S^{(r), \Theta}(Q)$ then
$K_\varepsilon(Z_{p_m} \cap Q)$ and $K_\varepsilon(Z_p \cap Q)$
must be close to \emph{some} member of $\mathcal{E}_0^\Theta$
and
thus, for $m$ sufficiently large, 
they
are close to the {\em same} member of $\mathcal{E}^{\Theta}_0$, which 
must be $E(p)$. In particular, for $m$ sufficiently large, we have $E(p_m) = E(p)$ and consquently 
$E_\alpha(p_m) = E_\alpha(p)$ for all $\alpha$.
\end{proof}
\noindent
(It is at the end of the proof of this lemma, and in the construction of the sets $S^{(r), \Theta}_\alpha(Q)$, 
where the relevance of the earlier decomposition into colours becomes clear.)

\medskip
\noindent
So, for $m$ sufficiently large we have
$$ \mbox{ vol }(\{p_m> 0\} \; \cap \; E_\alpha (p)) < \mbox{ vol}(\{p_m <0 \}
\; \cap \; E_\alpha(p))$$
which, upon taking limits and using the fact that vol $(\{ p = 0\}) = 0$ as $p$ is non-zero,
implies
$$ \mbox{ vol }(\{p > 0\} \; \cap \; E_\alpha (p)) \leq \mbox{ vol}(\{p < 0 \}
\; \cap \; E_\alpha(p)),$$
which is in contradiction with \eqref{last}. Hence
$$S^{(r), \Theta +}_\alpha(Q) \cap \overline{S^{(r), \Theta -}_\alpha(Q)} = \emptyset,$$
and we are therefore finished.

\section{Appendix -- Bisecting balls}\label{app}

\medskip
\noindent
In this appendix we indicate a simple proof of the (geometrically obvious) fact that a hypersurface
which bisects the unit ball must have large surface area inside the
ball. Let $\mathbb{B}$ be the closed unit ball in $\mathbb{R}^n$ and suppose $p: \mathbb{R}^n \to
\mathbb{R}$ is a polynomial. Let
$$E = \{x \in \mathbb{B} \, : \, p(x) \leq 0 \}$$
and 
$$F = \{x \in \mathbb{B} \, : \, p(x) \geq 0 \}.$$

\begin{lemma}\label{bisect}
If $ {\rm vol} \, (E) = a \, {\rm vol} \, (\mathbb{B})$ and
$ {\rm vol} \, (F) = b \, {\rm vol} \, (\mathbb{B})$ 
where $a + b = 1$, then 
$$\mathcal{H}_{n-1} (\{ x \in \mathbb{B} \, : \, p(x) = 0 \}) 
>  \frac{1}{2}\left(a^{(n-1)/n} + b^{(n-1)/n} -1 \right)
\mathcal{H}_{n-1}(\mathbb{S}^{n-1}).$$
\end{lemma}

\begin{proof}
It is easy to see that 
$ \partial E \cup \partial F = \mathbb{S}^{n-1} \cup (E \cap F)$
and
$ \partial E \cap \partial F = E \cap F.$
Since
$$ \mathcal{H}_{n-1}( \partial E \cup \partial F) = 
\mathcal{H}_{n-1}( \partial E) + \mathcal{H}_{n-1}(\partial F) 
- \mathcal{H}_{n-1}( \partial E \cap \partial F)$$
and 
$$ \mathcal{H}_{n-1} (\mathbb{S}^{n-1} \cup (E \cap F)) =
\mathcal{H}_{n-1} (\mathbb{S}^{n-1}) + \mathcal{H}_{n-1} (E \cap F)
- \mathcal{H}_{n-1} (\mathbb{S}^{n-1} \cap E \cap F)$$
we have
$$ 2 \mathcal{H}_{n-1} (E \cap F)
= \mathcal{H}_{n-1}( \partial E) +
\mathcal{H}_{n-1}(\partial F) -  \mathcal{H}_{n-1} (\mathbb{S}^{n-1})
+ \mathcal{H}_{n-1}
(\mathbb{S}^{n-1} \cap E \cap F),$$
and so
$$\mathcal{H}_{n-1} (E \cap F) 
\geq \frac{1}{2}\left(\mathcal{H}_{n-1}( \partial E) +
\mathcal{H}_{n-1}(\partial F) -  \mathcal{H}_{n-1} (\mathbb{S}^{n-1})\right).$$

\medskip
\noindent
By the isoperimetric inequality we have
$$\mathcal{H}_{n-1}( \partial E) \geq a^{(n-1)/n}
\mathcal{H}_{n-1}(\mathbb{S}^{n-1})$$
and 
$$\mathcal{H}_{n-1}( \partial F) \geq b^{(n-1)/n}
\mathcal{H}_{n-1}(\mathbb{S}^{n-1}),$$
(with strict inequality in at least one place) so that 
$$\mathcal{H}_{n-1} (E \cap F) 
> \frac{1}{2}\left(a^{(n-1)/n} + b^{(n-1)/n} -1 \right)
\mathcal{H}_{n-1}(\mathbb{S}^{n-1})$$
as required. 
\end{proof}

\medskip
\noindent
Since for $n \geq 2$ and $0 < a, b < 1$ with $a+b =1$ we have $ (a+b)^{(n-1)/n} < a^{(n-1)/n} + b^{(n-1)/n}$,
this establishes the desired bound.
\medskip
\noindent
 
\medskip
\noindent
The following question may be of interest. 
Let $K \subseteq \mathbb{R}^n$ be a symmetric convex body, which we can normalise so that its John ellipsoid 
is the unit ball. Within the class of polynomial hypersurfaces $Z$ which cut $K$ in the proportions $a:b$, 
how do we minimise the surface area of $Z \cap K$?

\bibliographystyle{plain}

\begin{thebibliography}{MMMMMM}

\bibitem{B} J.~Bennett, A trilinear restriction problem for the
  paraboloid in $\mathbb{R}^3$,
  Electron. Res. Announc. Amer. Math. Soc. 10 (2004), 97 -- 102. 

\bibitem{BCT} J.~Bennett, A.~Carbery and T.~Tao, On the multilinear restriction and Kakeya 
conjectures, Acta Mathematica 196, 2 (2006), 261 -- 302. 

\bibitem{B1} J.~Bourgain, Moment inequalities for trigonometric
  polynomials with spectrum in curved hypersurfaces, arXiv:1107.1129v1 [math.CA]

\bibitem{B2} J.~Bourgain, On the Schrodinger maximal function in
  higher dimension, arXiv:1201.3342v1 [math.AP]

\bibitem{BG} J.~Bourgain and L.~Guth, Bounds on oscillatory integral
  operators based on multilinear estimates, Geom. Funct. Anal. 21
  (2011), no. 6, 1239 -- 1295.

\bibitem{B3} J.~Bourgain, P.~Shao, C.D.~Sogge and X.~Yao On $L^p$-resolvent estimates and the density of eigenvalues for compact Riemannian manifolds, arXiv:1204.3927v2 [math.AP]

\bibitem{C} A.~Carbery, The Brouwer fixed point theorem and the
  Borsuk--Ulam theorem, arXiv:1205.4540v01 [math.AT]

\bibitem{CV} A.~Carbery and S.I.~Valdimarsson, Multilinear duality, manuscript in preparation.

\bibitem{Dv} Z.~Dvir, On the size of Kakeya sets in finite fields, J. Amer. Math. Soc. 22 (2009), 1093 -- 1097. 

\bibitem{G} L.~Guth, The endpoint case of the Bennett--Carbery--Tao multilinear Kakeya conjecture, 
Acta Mathematica 205 (2) (2010), 263 -- 286.  

\bibitem{John} F.~John, Extremum problems with inequalities as subsidiary conditions, Studies and essays presented to R. Courant on his 60th birthday, January 8, 1948, 187 -- 204. Interscience Publishers, Inc., (New York NY), 1948. 

\bibitem{Mat} J.~Matou\v{s}ek, Using the Borsuk--Ulam theorem, Springer (Heidelberg), 2003.

\end{thebibliography}

\end{document}